\documentclass[12pt]{amsart}
\usepackage{amsmath,amsthm,enumerate,graphics,pstricks,amsfonts,latexsym,amsopn,verbatim,amscd,amssymb}
\usepackage{color}
\usepackage{psfrag}
\usepackage[all]{xy}

\theoremstyle{plain}
\newtheorem{thm}{Theorem}[section]

\newtheorem{lem}[thm]{Lemma}
\newtheorem{cor}[thm]{Corollary}
\newtheorem{prop}[thm]{Proposition}

\theoremstyle{definition}
\newtheorem{rem}[thm]{Remark}
\newtheorem{fact}[thm]{Fact}
\newtheorem{exam}[thm]{Example}

\begin{document}

\title[On reduced zero-divisor graphs of posets]{On reduced zero-divisor graphs of posets}

\author[A. K. Das  and  D. Nongsiang]{Ashish Kumar Das*  }

\address{A. K. Das, Department of Mathematics, North-Eastern Hill University,
Permanent Campus, Shillong-793022, Meghalaya, India.}

\email{akdasnehu@gmail.com}

\author[]{  Deiborlang Nongsiang}
\address{D. Nongsiang, Department of Mathematics, North-Eastern Hill University,
Permanent Campus, Shillong-793022, Meghalaya, India.}

\email{ndeiborlang@yahoo.com}

\begin{abstract}
 In this paper we study some of the basic properties of a graph which is constructed from the equivalence classes of non-zero  zero-divisors determined by annihilator ideals of a poset.  In particular, we  demonstrate how this graph helps in identifying the annihilator prime ideals of  a poset that satisfies the ascending chain condition for its proper annihilator ideals. 
\end{abstract}


\thanks{*Corresponding author}
\maketitle

\section{Introduction} \label{S:intro}
The study of the interrelationship between algebra and graph theory by associating a graph to an algebraic object was initiated, in 1988, by 
I. Beck \cite{ib} who developed the notion of a zero divisor graph of a commutative ring with identity. Since then, a number of authors  have studied various forms of zero divisor graphs associated to rings and other algebraic structures (see, for example, \cite{an, al, nwd, hl}). 

 In 2009, R. Halas and M. Jukl \cite{hj} introduced the notion of a zero-divisor graph of a partially ordered set (in short, a poset). The study of various types of zero-divisor graphs of posets was then carried out by many others in \cite{xl, vj, admpy}. However, the zero-divisor graph of a poset considered in this paper was actually introduced by D. Lu and T. Wu \cite{lw}, which is slightly different from the one introduced in \cite{hj}.

In this paper, inspired by the ideas of Mulay \cite{mu} and Spiroff et al.\cite{sw}, we  study some of the basic properties of a graph which is constructed from the equivalence classes of non-zero  zero-divisors determined by annihilator ideals of a poset. This graph is same as the reduced graph of the zero-divisor graph of a poset (see \cite[page 798]{lw}). In particular, we  demonstrate how this graph helps in identifying the annihilator prime ideals of  a poset that satisfies the ascending chain condition for its proper annihilator ideals.

\section{Prerequisites}\label{S:pre}

In this section, we put together some well-known concepts, most of which can be found in \cite{co,dp,rh,rh2}.

We begin by recalling some of the basic terminologies from the theory of graphs. Needless to mention that all graphs considered here are
 simple graphs, that is, without loops or multiple edges. Let $G$ be a graph and $x,y \in V(G)$, the vertex set of $G$.  Then, $x$ and $y$ are said to be {\it adjacent} if $x\neq y$ and there is an edge $x-y$ between $x$ and $y$. A {\it walk} between $x$ and $y$ is a sequence  of adjacent vertices, often written as $x- x_1-x_2 - \dots -x_n-y$. A walk between $x$ and $y$ is called \textit{path} if the vertices  in it are all distinct (except, possibly, $x$ and $y$). A path between $x$ and $y$ is called a {\it cycle} if $x=y$. The number of edges in a path or a cycle, is called its {\it length}. If $x \neq y$, then the   minimum of the lengths of all paths between $x$ and $y$ in $G$ is called the \textit{distance} between $x$ and $y$ in $G$, and is denoted by ${\rm dist}(x,y)$. If there is no path between $x$ and $y$, then we define ${\rm dist}(x,y)=\infty$. The maximum of all possible distances in $G$ is called the \textit{diameter} of $G$, and is denoted by ${\rm diam}(G)$.  The {\it girth} of a graph $G$ is the minimum of the lengths of all cycles in $G$, and is denoted by ${\rm girth}(G)$. If $G$ is {\it acyclic}, that is, if $G$ has no cycles, then we write ${\rm girth}(G)=\infty$.  A \textit{cycle graph} is a graph that consists of a single cycle (an \textit{n}-gon). 
 
A subset of the vertex set of a graph $G$ is called a \textit{clique} of $G$ if it consists entirely of pairwise adjacent vertices. The least upper bound of the sizes of all the cliques of $G$ is called the \textit{clique number} of $G$, and is denoted by $\omega (G)$.

The \textit{neighborhood} of a vertex $x$ in a graph $G$, denoted by ${\rm nbd}(x)$, is defined to be the set of all vertices adjacent to $x$  while the {\it degree} of $x$ in $G$, denoted by $\deg(x)$, is defined to be the number of vertices adjacent to $x$, and so $\deg(x)= |{\rm nbd}(x)|$. If $\deg(x)=1$, then  $x$ is said to be an {\it end vertex} in $G$. If $\deg(x)=\deg(y)$ for all $x,y \in V(G)$, then the graph $G$ is said to be a \textit{regular graph}. 

A graph $G$ is said to be {\it connected} if there is a path between every
pair of distinct vertices in $G$. A graph $G$ is said to be {\it complete} if there is an edge between every pair of distinct vertices in $G$. We denote the complete graph with $n$ vertices by $K_n$. An {\it r-partite graph}, $r \geq 2$, is a graph whose vertex set can be partitioned into $r$ disjoint parts in such a way that no two adjacent vertices lie in the same part. Among the $r$-partite graphs, the {\it complete r-partite graph} is the one in which two vertices are adjacent if and only if they lie in different parts. The complete $r$-partite graph with parts of size $n_1$,$n_2$, $\dots$, $n_r$ is denoted by $K_{n_1,n_2, \dots, n_r}$. The $2$-partite and  complete $2$-partite graphs are popularly known as \textit{bipartite} and \textit{complete bipartite} graphs respectively.  A bipartite graph of the form $K_{1,n}$ is also known as a \textit{star graph}.

Next we turn to partially ordered sets and their zero-divisor graphs. A non-empty set is said to be a \textit{partially ordered set}  (in short, a \textit{poset})  if it is equipped with  a \textit{partial order}, that is, a reflexive, anti-symmetric and transitive binary relation. It is customary to denote a partial order by `$\leq$'.

Let $Q$ be a non-empty subset of a poset $P$. If there exists $y \in Q$ such that $y \leq x$ for every $x \in Q$, then $y$ is called the \textit{least element} of $Q$. The least element of $P$,  if exists, is usually denoted by 0.  An element $x \in Q$ is called a \textit{minimal element} of $Q$ if $y \in Q$ and $y \leq x$ imply that $y = x$. We denote the set of  minimal elements of $Q$ by Min($Q$).

Let $P$ be a poset with least element 0.  An element $x \in P$ is called a \textit{ zero-divisor} of $P$ if there exists $y \in P^\times := P\setminus\{0\}$ such that the set $L(x,y):= \{z\in P\mid z\leq x\
{\rm and}\ z\leq y\} = \{0\}$. We denote the set of zero-divisors of $P$ by $Z(P)$ and write $Z(P)^\times := Z(P) \setminus \{0\}$.  By an \textit{ideal} of $P$ we mean a non-empty subset $I$ of $P$ such that $y \in I$ whenever $y \leq x$ for some $x \in I$. We say that the ideal $I$ is \textit{proper} if $I \neq P$. For each $x \in P$,  it is easy to see that the set $(x]:=\{y\in P\mid y\leq x\}$ is an ideal of $P$,  called the \textit{ principal ideal} of $P$ generated by $x$. Given $x \in P$, the \textit{annihilator} of $x$ in $P$ is defined to be the set ${\rm ann}(x):= \{y\in P\mid L(x,y)=\{0\}\}$, which is also an ideal of $P$.  Note that $x \notin {\rm ann}(x)$ for all $x \in P^\times$.  A proper ideal $\mathfrak{p}$ of $P$ is called a \textit{prime ideal} of $P$ if for every $x, y \in P$, $L(x,y) \subseteq \mathfrak{p}$ implies that either $x \in \mathfrak{p}$ or $y \in \mathfrak{p}$. A   prime ideal $\mathfrak{p}$ of $P$ is said to be an \textit{annihilator prime ideal}  (or, an \textit{associated prime}) if there exists $x \in P$ such that $\mathfrak{p} = {\rm ann}(x)$.  Two annihilator prime ideals  ${\rm ann}(x)$ and ${\rm ann}(y)$ are distinct if and only if   $L(x,y)=\{0\}$ (see  \cite[Lemma 2.3]{hj}). We write   ${\rm Ann}(P)$ to denote the set of all annihilator prime ideals of $P$.

Let $P$ be a poset with least element 0 and with $Z(P)^\times \neq \emptyset$. 
 As  in \cite{lw},  the \textit{zero-divisor graph} of  $P$  is defined to be the graph $\Gamma(P)$ in which the vertex set is $Z(P)^\times$, and two vertices $x$ and $y$ are adjacent if and only if $L(x,y) = \{0\}$. Clearly,  $\Gamma(P)$ is a simple graph; in fact, ${\rm nbd}(x) = {\rm ann}(x) \setminus \{0\}$ for all $x \in V(\Gamma(P))$. It is well-known  that $\Gamma(P)$ is a connected graph with ${\rm diam}(\Gamma(P))\in\{1,2,3\}$   and  ${\rm girth}(\Gamma(P))\in\{3,4,\infty\}$ (see, for example, \cite{admpy}). The clique number $\omega(\Gamma(P))$ of $\Gamma(P)$ is usually denoted, in short, by $\omega (P)$. Clearly, $\omega (P) \geq 2$. It may be  noted here that $\omega (P) = {\rm clique}(P)-1$, where ${\rm clique}(P)$ is the clique number of the zero-divisor graph considered in \cite{hj} whose vertex set is the whole of $P$.

Finally, we would like to mention that all posets considered in this paper are with least element $0$ and have non-zero zero-divisors, unless explicitly written otherwise.

\section{Reduced graph of $\Gamma (P)$: basic properties} 

In this section, we study some of the basic properties of the reduced graph of the zero-divisor graph of a poset.

Let $P$ be a poset with least element 0 and with $Z(P)^\times \neq \emptyset$. Given $x,y\in P$, set $x\sim y$ if  ann($x$)=ann($y$). Clearly, $\sim$ is an equivalence relation in $P$. Let $[x]$ denote the equivalence class of $x \in P$. Note that if $x\in Z(P)^\times$, then $[x]\subseteq Z(P)^\times$. Also, note that $[0]=\{0\}$  and $[x]=P\setminus Z(P)$ for all $x\in P\setminus Z(P)$. In  analogy with \cite{mu, sw},  the \textit{graph of equivalence classes of zero divisors} of $P$ may be defined to be the graph $\Gamma_E(P)$ in which the vertex set is the set of all equivalence classes of the elements of $Z(P)^\times$, and two vertices $[x]$ and $[y]$ are adjacent if and only if $L(x,y)=\{0\}$, that is, if and only if $x$ and $y$ are adjacent in $\Gamma(P)$. Note that two adjacent vertices in $\Gamma(P)$ represent two distinct equivalence classes, and hence, two distinct vertices in $\Gamma_E(P)$. Thus $\Gamma_E(P)$ is also a simple graph. It may be recalled (see \cite{sw}) that, in case of a commutative ring $R$ with  unity, two adjacent vertices in $\Gamma(R)$ do not necessarily represent two distinct vertices in $\Gamma_E(R)$. Since ${\rm nbd}(x) = {\rm ann}(x) \setminus \{0\}$ for all $x \in V(\Gamma(P))$,  one obtains the same graph $\Gamma_E(P)$ if the equivalence relation $\sim$ is defined on $V(\Gamma (P))=Z(P)^\times$ by setting $x\sim y$, where  $x,y\in V(\Gamma (P))=Z(P)^\times$, if and only if  nbd($x$)=nbd($y$) in $\Gamma (P)$. Thus, the graph  $\Gamma_E(P)$ is same as the reduced graph of $\Gamma (P)$ considered and characterized by Lu et al. in \cite{lw}.  In particular, $\Gamma_E (P)$ is also a zero-divisor graph of some poset, and hence, there are lots of structural similarities between  $\Gamma (P)$ and $\Gamma_E (P)$.
 There are however some more features of $\Gamma_E (P)$ which are worth looking into.

The graph $\Gamma_E(P)$ has some advantages over the zero divisor graph $\Gamma(P)$. In many cases $\Gamma_E(P)$ is finite when $\Gamma(P)$ is infinite. For example, consider the poset $P_0 =\{\emptyset,\{1\},\{2\},\{2,3\},\{2,4\},\dots\}$ under the set inclusion with least element $\emptyset$.  Then $\Gamma(P_0)$ is infinite, whereas $\Gamma_E(P_0)$ is finite and has only two vertices. In fact, the zero-divisors $\{2\}, \{2,3\}, \{2,4\},\dots,$ have the same annihilator and so they represent a single vertex in $\Gamma_E(P_0)$; the other vertex in $\Gamma_E(P_0)$ is represented by $\{1\}$. Another important aspect of $\Gamma_E(P)$ is its connection to the annihilator prime ideals of the poset $P$, which we discuss in detail in the next section. Let us now rewrite a fact, just noted above, in a more explicit manner for the ease of its extensive use (often without a mention) in this paper.

\begin{fact}\label{fact3.1}
Given a poset $P$,  two  vertices in $\Gamma(P)$ are adjacent  if and only if they represent two adjacent vertices in $\Gamma_E(P)$.
\end{fact}

 In view of Fact \ref{fact3.1}, it is easy to see that, given a poset $P$, we have $\omega (P) = \omega (\Gamma_E (P))$, that is, the clique numbers of $\Gamma(P)$ and $\Gamma_E(P)$  are same.  
 
By \cite[Corollary 3.3 (2)]{lw}, we know that $\Gamma_E(P)$ is also a zero-divisor graph of some poset. Therefore, in view of  \cite[Theorem 3.3]{admpy}, $\Gamma_E (P)$ is a  connected graph with ${\rm diam} \Gamma_E (P)\leq 3$. Our first result of this section not only generalizes Fact \ref{fact3.1} but also shows some similarity between $\Gamma (P)$ and $\Gamma_E (P)$ as far as their diameter is concerned.
\begin{prop} \label{prop3.2} Let $P$ be a poset. If $[x]$ and $[y]$ are two distinct vertices in $\Gamma_E(P)$, then  ${\rm dist}([x],[y]) =   {\rm dist}(x,y)$. Moreover, the following assertions hold:
\begin{enumerate} 
\item  ${\rm diam} (\Gamma_E(P))=3$ if and only if ${\rm diam} (\Gamma (P))=3$.
\item  ${\rm diam} (\Gamma_E(P))=2$ if and only if ${\rm diam} (\Gamma (P))=2$  and  ${\rm ann}(x) \neq  {\rm ann}(y)$ for some $x,y \in Z(P)^\times$ with $L(x,y)\neq \{0\}$.
\item  ${\rm diam} (\Gamma_E(P))=1$ if and only if ${\rm ann}(x) =  {\rm ann}(y)$ for  all $x,y \in Z(P)^\times$ with $L(x,y)\neq \{0\}$.
\end{enumerate}
\end{prop}

\begin{proof}
Let $[x]$ and $[y]$ be two distinct vertices in $\Gamma_E(P)$. If $[x] - [x_1]  - \dots - [x_n] - [y]$ is a path in $\Gamma_E(P)$ then, by Fact \ref{fact3.1}, $x - x_1  - \dots - x_n - y$ is a path in $\Gamma (P)$. Therefore, ${\rm dist}([x],[y]) \geq   {\rm dist}(x,y)$. On the other hand, if $x - x_1  - \dots - x_n - y$ is a path in $\Gamma (P)$, then, once again using Fact \ref{fact3.1},  $[x] - [x_1]  - \dots - [x_n] - [y]$ is a walk in $\Gamma_E(P)$. It follows that ${\rm dist}([x],[y]) \leq  {\rm dist}(x,y)$, and hence, the equality holds.

For proving the given assertions, we first note that there exist two distinct vertices $[x]$ and $[y]$ of $\Gamma_E (P)$ such that ${\rm dist}([x],[y]) = {\rm diam} (\Gamma_E (P))$. Therefore, it follows from the first half of this proposition that  we always have ${\rm diam} (\Gamma_E (P)) \leq {\rm diam} (\Gamma (P))$. However,  for the reverse inequality it is not enough to have two distinct vertices $x,y \in V(\Gamma (P))$ such that ${\rm dist}(x,y) = {\rm diam} (\Gamma (P))$; we must also have an additional requirement, namely, ${\rm ann}(x)\neq {\rm ann}(y)$. If ${\rm dist}(x,y) = 3$, then this additional requrement is guaranteed by the existence of a path of the form $x-a-b-y$ in $\Gamma (P)$. Hence, the assertion  (a) follows. This in turn also proves the assertion  (b), because the extra condition included in (b) fulfills the additional requirement mentioned above. The last assertion, namely, (c) now follows from the assertions  (a) and (b).\end{proof}

From \cite{admpy} and \cite{lw}, we know that ${\rm girth}(\Gamma_E(P)) \in \{3,4, \infty\}$. In this context, we have the following result.
\begin{prop}\label{prop3.3} Let P be a poset. Then,  
\[
{\rm girth}(\Gamma_E(P))= 3 \Longleftrightarrow {\rm girth}(\Gamma(P))= 3  \Longleftrightarrow |V(\Gamma_E(P))| \geq 3.
\]
Consequently, ${\rm girth}(\Gamma_E(P))=\infty$ if and only if $|V(\Gamma_E(P))|=2$.
\end{prop}

\begin{proof} Assume that $|V(\Gamma_E(P))| \geq 3$.  Then we also have              $|V(\Gamma (P))| \geq 3$. Moreover, $\Gamma (P)$ is not a star graph; otherwise we would have $|V(\Gamma_E(P))| = 2$. By  \cite[Theorem 4.2]{admpy}, we have  ${\rm girth}(\Gamma(P)) \in \{3, 4,\infty\}$. If girth$(\Gamma(P))=4$ or $\infty$, then it follows from   \cite[Remark 4.12]{admpy} that $\Gamma(P)$ is a complete bipartite graph, which in turn implies that $|V(\Gamma_E(P))| = 2$. Therefore, we have girth$(\Gamma(P))=3$. Hence, in view of Fact \ref{fact3.1}, we have girth$(\Gamma_E(P))=3$. On the other hand, if ${\rm girth}(\Gamma_E(P))= 3$ then we obviously have $|V(\Gamma_E(P))| \geq 3$.

The consequential statement is clear as, by the choice of $P$, we always have $|V(\Gamma_E(P))| \geq 2$. \end{proof}

As an immediate consequence, we have the following corollary which may be compared with \cite[Proposition 1.8]{sw}.

\begin{cor}\label{cor3.4} There is no poset P for which $\Gamma_E(P)$ is a cycle graph with at least four vertices.\end{cor}

Our next result is a small  observation, which has some interesting consequences in contrast to some results of similar nature in \cite{sw}.

\begin{prop}\label{prop3.5} Let $P$ be a poset. Then no two distinct vertices of $\Gamma_E(P)$ have the same neighborhood. Equivalently, given $[x], [y] \in V(\Gamma_E(P))$, one has ${\rm nbd}([x])= {\rm nbd}([y])$ if and only if ${\rm ann}(x)={\rm ann}(y)$.  \end{prop}

\begin{proof} Let $[x]$ and $[y]$ be two distinct vertices in $\Gamma_E(P)$  such that ${\rm nbd}([x])= {\rm nbd}([y])$. Let $z \in {\rm ann}(x)$, $z\neq 0$. Then, by Fact \ref{fact3.1}, we have   $[z]\in {\rm nbd}([x])= {\rm nbd}([y])$, whence $z\in {\rm ann}(y)$. Thus, it follows that ${\rm ann}(x)\subseteq {\rm ann}(y)$. Similarly, we have  ${\rm ann}(y)\subseteq {\rm ann}(x)$, and so ${\rm ann}(x)={\rm ann}(y)$. This contradiction proves the proposition. \end{proof} 

The above proposition says, in other words, that the reduced zero-divisor graph of a poset cannot be further reduced  (symbolically, $\Gamma_E (P) \cong (\Gamma_E)_E (P)$).
\begin{cor}\label{cor3.6} Let P be a poset and  $r\geq 2$.  Then,  $\Gamma_E(P)$ is a complete r-partite graph  if and only if it is a complete graph with r-vertices.\end{cor}

\begin{proof} Let $\Gamma_E(P)$ be a complete r-partite graph with  $V(\Gamma_E(P))=V_1\sqcup V_2\sqcup \dots \sqcup V_r$. Let $[x], [y] \in V_j$, where $1\leq j \leq r$. Then, ${\rm nbd}([x])= {\rm nbd}([y])= V(\Gamma_E(P))\setminus V_j$, and so, by Proposition \ref{prop3.5}, $[x]=[y]$. Thus, $\mid V_j \mid=1$ for all $j \in \{1,2, \dots, r\}$, which means that $\Gamma_E(P)$ is a complete graph with \textit{r} vertices. The converse is trivial. \end{proof}

\begin{cor}\label{cor3.7} Let $P$ be a poset such that $|V(\Gamma_E (P))| \geq 3$.  Then $\Gamma_E(P)$ is not a star graph.\end{cor}

\begin{proof}
It is enough to note that in a star graph with at least three vertices all the end vertices have the same neighborhood. Alternatively, one may also note that star graphs are complete bipartite graphs. \end{proof}

\begin{cor}\label{cor3.8} There is only one graph with exactly three vertices  that can be realized as the graph $\Gamma_E(P)$ for some poset P, and it is the cycle/complete graph $K_{3}$.\end{cor}

\begin{proof} Since, given a poset P, the graph $\Gamma_E(P)$ is a connected but not a star graph, we need only to note that the  graph of equivalence classes of zero divisors of the poset $\{\emptyset,\{1\},\{2\},\{3\}\}$ is precisely $K_{3}$.\end{proof}

One of the obvious consequences of Fact \ref{fact3.1} is that if for some poset $P$ the graph $\Gamma (P)$ is complete, then the corresponding graph $\Gamma_E(P)$ is also complete; in fact, $\Gamma_E(P) \cong \Gamma (P)$.  Also, it follows from the definition of $\Gamma_E(P)$ that if $\Gamma (P)$ is a complete $r$-partite graph, $r \geq 2$, then $\Gamma_E(P)$ is a complete graph with $r$ vertices. Thus, for every $n \geq 2$, the complete graph $K_n$ can be realized as  $\Gamma_E(P)$ for some poset $P$; for example, we may consider the poset $\{\emptyset,\{1\},\{2\}, \dots, \{n\}\}$.  This is not the case for rings (see \cite[Proposition 1.5]{sw}).

\section{Posets with an ascending chain condition}

In this section, imposing certain restrictions on a given poset,  we study some properties of its reduced zero-divisor graph in terms of the annihilator prime ideals.

Let $P$ be a poset. We say that $P$ is \textit{a poset with ACC for annihilators} if the ascending chain condition holds for its annihilator ideals, that is, if there is no  infinite strictly  ascending chain in the set $\mathfrak{A} := \{ann(x) \mid x \in P^\times \}$ under set inclusion. Equivalently, $P$ is a poset with ACC for annihilators  if and only if every non-empty subset of $\mathfrak{A}$ has a maximal element. Thus, if $P$ is a poset with ACC for annihilators, then $\mathfrak{A}$ has a maximal element and every element of $\mathfrak{A}$ is contained in a maximal element of $\mathfrak{A}$. We  denote the set of all maximal elements of $\mathfrak{A}$ by ${\rm Max} (\mathfrak{A})$.

If $P$ is a poset such that $\omega (P) < \infty$, then from \cite[Lemma 2.4]{hj} it follows that $P$ is a poset with ACC for annihilators. In particular, if $P$ is a poset such that  ${\rm deg}(x) < \infty$ for all $x \in V(\Gamma (P))$, then $P$ is a poset with ACC for annihilators; noting that $\Gamma (P)$ has no infinite clique, and so, by \cite[Lemma 2.10]{hj}, we have $\omega (P) < \infty$.   On the other hand, consider the poset \[P= \{\emptyset\}\cup \{\{n\}\mid n\in \mathbb{N}\}\cup\{\{m,m+1, \dots\}\mid m\in \mathbb{N},m \neq 1\}\] under set inclusion, where $\mathbb{N}$ denotes the set of all positive integers. It is easy to see that there is an infinite strictly ascending chain of annihilator ideals of $P$ given by
\[
{\rm ann}(\{2,3, \dots\})\subsetneq {\rm ann}(\{3,4, \dots\})\subsetneq  \dots,
\]
which means that $P$ is a poset without ACC for annihilators. 
\begin{rem} \label{rem4.1}
If $P$ is a poset with ACC for annihilators, then the clique number of $\Gamma (P)$ need not be finite, that is, $\Gamma (P)$ may contain an infinite clique; for example, consider the poset $P= \{\emptyset\}\cup \{\{n\}\mid n\in \mathbb{N}\}$ under set inclusion. This is contrary to what has been asserted in Proposition 2.6 of \cite{lw}. In fact, a careful look at the proof of \cite[Proposition 2.6]{lw} reveals that while proving $``(2) \Rightarrow (3)"$ the authors mistakenly assumed the validity of the first statement of the said proposition. 
\end{rem}
  
  Our first result concerning the set of all annihilator prime ideals of a poset is  given as follows.

\begin{prop}\label{prop4.2} Let $P$ be a poset.  Then, \[{\rm Ann}(P) = {\rm Max} (\mathfrak{A}). \]
In particular, if $P$ is a poset with ACC for annihilators, then ${\rm Ann}(P) \neq \emptyset$.
\end{prop}

\begin{proof} By \cite[Lemma 2.2]{hj}, we  have  ${\rm Max} (\mathfrak{A})  \subseteq  {\rm Ann}(P)$.  Conversely, suppose that ${\rm ann}(x) \in {\rm Ann}(P) \setminus {\rm Max} (\mathfrak{A}) $. Then there exists ${\rm ann}(y) \in \mathfrak{A}$ such that ${\rm ann}(x) \subsetneq {\rm ann}(y)$.  Choose $z \in {\rm ann}(y) \setminus {\rm ann}(x)$. Since  $L(z,y)=\{0\} \subseteq {\rm ann}(x)$ and ${\rm ann}(x)$ a prime ideal, it follows that $y \in {\rm ann}(x) \subset {\rm ann}(y)$ which is absurd. Hence, we have ${\rm Ann}(P) = {\rm Max} (\mathfrak{A})$. The particular case follows from the fact that if $P$ is a poset with ACC for annihilators, then  ${\rm Max} (\mathfrak{A}) \neq \emptyset$.\end{proof}

Given a poset $P$, consider the set $\mathfrak{B} := \{ann(x) \mid x \in Z(P)^\times \}$. Clearly, $\emptyset \neq \mathfrak{B} \subseteq \mathfrak{A}$; in fact, $\mathfrak{B} = \mathfrak{A} \setminus \{0\}$. Since $Z(P)^\times \neq \emptyset$, $\{0\}$ is not a prime ideal of $P$, and so, it follows that ${\rm Ann}(P) \subseteq \mathfrak{B}$.  Note that there is a natural bijective map from $\mathfrak{B}$ to the vertex set of $\Gamma_E (P)$ given by ${\rm ann}(x) \mapsto  [x]$. As such, we may treat $\mathfrak{B}$ as the vertex set of $\Gamma_E (P)$.  In view of this, with a slight abuse of terminology, we sometimes refer to  $[x] \in V(\Gamma_E(P))$ as an annihilator ideal (respectively, an annihilator prime ideal) if we have ${\rm ann}(x) \in \mathfrak{B}$ (respectively,   ${\rm ann}(x) \in {\rm Ann}(P)$). All the forthcoming results of this section are under this identification.  

The following lemma plays a crucial role in this section.

\begin{lem}\label{lem4.3}
Let $P$ be a poset. Then the following assertions hold:
\begin{enumerate} 
\item  Given ${\rm ann}(x), {\rm ann}(y) \in \mathfrak{B}$, one has ${\rm ann}(x)  \subsetneq {\rm ann}(y)$ if and only if  ${\rm nbd}([x]) \subsetneq {\rm nbd}([y])$ in $\Gamma_E (P)$. 
\item  Given ${\rm ann}(x) \in \mathfrak{B}$ and ${\rm ann}(z) \in {\rm Ann}(P)$, one has ${\rm ann}(x)  \nsubseteq {\rm ann}(z)$ if and only if $L(x,z)=\{0\}$; equivalently,  one has ${\rm nbd}([x]) \subseteq {\rm nbd}([z])$ if and only if the vertices  $[x]$  and $[z]$ are not  adjacent in $\Gamma_E (P)$.
\item Given  ${\rm ann}(x), {\rm ann}(y) \in \mathfrak{B}$,  if   ${\rm ann}(x) \cup {\rm ann}(y) \subseteq {\rm ann}(z)$ for  some  ${\rm ann}(z) \in {\rm Ann}(P)$, then   the vertices  $[x]$  and $[y]$ are not  adjacent in $\Gamma_E (P)$,  that is,  $L(x,y)\neq \{0\}$.  Converse is also true if $P$ is a poset with ACC for annihilators.
\item If $P$ is a poset with ACC for annihilators, then for each ${\rm ann}(x) \in \mathfrak{B}$ there exists ${\rm ann}(z) \in {\rm Ann}(P)$ such that ${\rm ann}(x)  \nsubseteq {\rm ann}(z)$.
\item  If ${\rm ann}(x) \in \mathfrak{B} \setminus {\rm Ann}(P)$, then there exist $u,v \notin {\rm ann}(x)$ such that $L(u,v) = \{0\}$.
\end{enumerate}
\end{lem}
\begin{proof}
(a)\; In view of Proposition \ref{prop3.5} and the definition of $\Gamma_E (P)$, it is enough to note that, given $u, x \in Z(P)^\times$, one has $u \in {\rm ann}(x)$ if and only if  $[u] \in {\rm nbd}([x])$ in $\Gamma_E (P)$.

 (b)\; If ${\rm ann}(x)  \nsubseteq {\rm ann}(z)$, then,   Choosing $y \in {\rm ann}(x) \setminus {\rm ann}(z)$, we have   $L(x,y)=\{0\} \subseteq {\rm ann}(z)$, and so $x \in {\rm ann}(z)$, since ${\rm ann}(z)$ is a prime ideal. Conversely, if $L(x,z)=\{0\}$, then $z \in {\rm ann}(x)$, and so ${\rm ann}(x)  \nsubseteq {\rm ann}(z)$, since $z \notin {\rm ann}(z)$.    The equivalent assertion follows from part (a).
 
  (c)\; If  ${\rm ann}(x) \cup {\rm ann}(y) \subseteq {\rm ann}(z)$ for some ${\rm ann}(z) \in {\rm Ann}(P)$, then, by part (b), we have $x, y \notin {\rm ann}(z)$, which implies that $L(x,y)\neq \{0\}$, since ${\rm ann}(z)$ is a prime ideal.
  Conversely, if  $L(x,y)\neq \{0\}$, then, choosing $w \in L(x,y)\setminus \{0\}$, we have ${\rm ann}(x) \cup {\rm ann}(y) \subseteq {\rm ann}(w) \subseteq {\rm ann}(z)$ for some ${\rm ann}(z) \in {\rm Ann}(P)$.
  
 (d)\; Let $y \in {\rm ann}(x) \setminus \{0\}$. Choose ${\rm ann}(z) \in  {\rm Ann}(P)$  such that ${\rm ann}(y)  \subseteq {\rm ann}(z)$. Since $x \in {\rm ann}(y)$, we have $x \in {\rm ann}(z)$, or equivalently, $z \in {\rm ann}(x)$.  Thus  ${\rm ann}(x)  \nsubseteq {\rm ann}(z)$, since $z \notin {\rm ann}(z)$.
 
  (e)\; If ${\rm ann}(x) \in \mathfrak{B} \setminus {\rm Ann}(P)$, then, by proposition \ref{prop4.2}, there exists ${\rm ann}(u) \in \mathfrak{A}$ such that ${\rm ann}(x) \subsetneq {\rm ann}(u)$, and so we may choose $v \in {\rm ann}(u) \setminus {\rm ann}(x)$ to complete the proof. \end{proof}

\begin{prop}\label{prop4.4} Let \ $P$ be a poset. Then,  the following assertions hold:
\begin{enumerate} 
\item  ${\rm Ann}(P)$ is a clique of $\Gamma_E(P)$.
\item  Given $[z] \in \mathfrak{B}$, one has $[z] \in {\rm Ann}(P)$ if and only if no two vertices in the set  $V(\Gamma_E (P)) \setminus {\rm nbd}([z])$ are adjacent. 
\item  If $P$ is a poset with ACC for annihilators, then every vertex in $\Gamma_E(P)$ is adjacent to an annihilator prime ideal; consequently, $|{\rm Ann}(P)| \geq 2$.
\end{enumerate}
\end{prop}
\begin{proof}
In view of Fact \ref{fact3.1}, the results follow from Lemma \ref{lem4.3}. More precisely, the first part follows from part (b) using  maximality of the annihilator prime ideals,  the second part from parts (b), (c) and (e),  and the third part from  parts (b) and (d) of Lemma \ref{lem4.3}. The first part also follows directly from \cite[Lemma 2.3]{hj}.\end{proof}

If $P$ is a poset with $\omega (P) < \infty$, then, in view of \cite[Lemmas 2.6]{hj}, it follows from Proposition \ref{prop4.2} that  $|{\rm Ann}(P)|  < \infty $. In this context, we have a stronger result in the following form.
\begin{prop}\label{prop4.5} Let  $P$ be a poset with ACC for annihilators. Then, 
\[|{\rm Ann}(P)| = \omega (P).\]
In particular, $|{\rm Ann}(P)| \geq 3$ if and only if $|V(\Gamma_E (P))|\geq 3$.
\end{prop}
\begin{proof}
By Proposition \ref{prop4.4}(a), we have $|{\rm Ann}(P)| \leq \omega(\Gamma_E (P)) = \omega (P)$. Conversely, suppose that  $\{x_1,x_2, \dots, x_k\}$  is a   clique of $\Gamma (P)$. For each $i \in \{1,2, \dots, k\}$, choose ${\rm ann}(z_i) \in {\rm Max} (\mathfrak{A})= {\rm Ann}(P)$ such that ${\rm ann}(x_i) \subseteq {\rm ann}(z_i)$.  By Lemma \ref{lem4.3}(c),  the annihilator prime ideals ${\rm ann}(z_1)$,  ${\rm ann}(z_2)$, $\dots$, ${\rm ann}(z_k)$ are all pairwise distinct.  It follows that $|{\rm Ann}(P)|$ is an upper bound for the sizes of all the cliques of $\Gamma (P)$, and hence,  $|{\rm Ann}(P)| \geq \omega (P)$.   This proves the first part.  The  particular case follows from Proposition \ref{prop3.3}. \end{proof}

It follows immediately from Proposition \ref{prop4.4}(a) and  Proposition \ref{prop4.5} that, given a poset $P$ with ACC for annihilators, ${\rm Ann}(P)$ is a clique of maximal size in $\Gamma_E (P)$.

\begin{rem}\label{rem4.6}
Let $P$ be a poset with ACC for annihilators. Then, as a trivial consequence of Proposition \ref{prop4.5}, we have $| V(\Gamma_E (P))| \geq 3$ if and only if  no annihilator prime ideal in $\Gamma_E (P)$ is an end vertex, and so, by Proposition \ref{prop4.4}(c), every vertex in $\Gamma_E (P)$ that is adjacent to an end vertex is an annihilator prime ideal (compare with  \cite[Proposition 3.2 and Corollary 3.3]{sw}).
\end{rem}

Given a poset $P$, if the degree of some annihilator prime ideal in $\Gamma_E (P)$ is infinity, then obviously $|V(\Gamma_E (P))|=\infty$. We have, however, a stronger converse given by the following result which may be compared and contrasted with \cite[Proposition 2.2]{sw}.
\begin{prop}\label{prop4.7}
Let  $P$ be a poset. If the vertex set of $\Gamma_E (P)$ is infinite, then the degree of each annihilator prime ideal in $\Gamma_E (P)$ is infinity.
\end{prop}
\begin{proof}
Suppose that $|V(\Gamma_E (P))|=\infty$. In view of Proposition \ref{prop4.4}(a), we may assume that $|{\rm Ann}(P)|< \infty$. In that case the set $\mathfrak{B} \setminus  {\rm Ann}(P)$ is infinite, that is, there are infinitely many vertices in   $\Gamma_E (P)$ which are not annihilator prime ideals. Let ${\rm ann}(z) \in {\rm Ann}(P)$ such that ${\rm deg}([z]) <\infty$. Then there exists an infinite set      $S \subseteq \mathfrak{B} \setminus  {\rm Ann}(P)$ such that no member of $S$ is adjacent to $[z]$ in $\Gamma_E (P)$. Therefore, by Lemma \ref{lem4.3}(b), we have  ${\rm nbd}([x])\subseteq {\rm nbd}([z])$ for all $[x] \in S$. Hence, using Proposition \ref{prop3.5}, it follows that the neighborhood  ${\rm nbd}([z])$ is an infinite set. This contradiction completes the proof.\end{proof}

Given a poset $P$, if there is a vertex $[z] \in V(\Gamma_E(P))$ such that ${\rm deg}([z]) > {\rm deg}([x])$ for all $[x] \in V(\Gamma_E(P)) \setminus \{[z]\}$, then it follows from Lemma \ref{lem4.3}(a) that  $[z]$ is a maximal element of $\mathfrak{A}$, that is, an annihilator prime ideal of $P$ (compare with \cite[Proposition 3.1]{sw}).  In this context we have a more general result  given as follows (compare with \cite[Proposition 3.6]{sw}).

\begin{prop}\label{prop4.8} Let P be a poset such that  $|V(\Gamma_E(P))|<\infty$. Then every vertex of maximal degree in $\Gamma_E(P)$ is a  maximal element of $\mathfrak{A}$, that is, an annihilator prime ideal of $P$.\end{prop}
\begin{proof}
If a vertex of maximal degree in $\Gamma_E(P))$ is not a  maximal element of $\mathfrak{A}$, then, using Lemma \ref{lem4.3}(a) and the condition that $|V(\Gamma_E(P))|<\infty$, we have a contradiction to the maximality of its degree. Hence the result follows.\end{proof}

 Note that the converse of the above proposition is false, that is, an annihilator prime ideal need not always be of maximal degree. For example, consider the poset    $\{\emptyset$, $\{1\}$, $\{2\}$, $\{3\}$, $\{4\}$, $\{5\}$, $\{6\}$, $\{1,2\}$, $\{3,4\}$, $\{3,5\}$, $\{3,6\}\}$ under set inclusion, in which ${\rm ann}(\{3\})$ is a prime ideal but ${\rm ann}(\{1,2\})$ is not; however, in the corresponding reduced zero-divisor graph, the degree of ${\rm ann}(\{3\})$ is $6$ and that of ${\rm ann}(\{1,2\})$ is $7$.

As an immediate consequence of Proposition \ref{prop4.8}, we have the following result.

\begin{cor}\label{cor4.9} Let \ P be a poset such that  $|V(\Gamma_E(P))|<\infty$.  Then $\Gamma_E (P)$ is a regular graph if and only if it is a complete graph.\end{cor}

\begin{proof} In a regular graph, every vertex is of maximal degree. Therefore, if $\Gamma_E (P)$ is a regular graph, then, by Proposition \ref{prop4.8}, $V(\Gamma_E(P))$ coincides with ${\rm Ann}(P)$. Thus, by Proposition     \ref{prop4.4}(a), $\Gamma_E (P)$ is a complete graph. The converse is trivial.\end{proof}

In the same context, it may be worthwhile to mention the following result.
\begin{prop}\label{prop4.10}
Let  P be a poset. If $\Gamma (P)$ is an $r$-regular graph, then $\Gamma_E (P)$ is a complete graph.\end{prop}
\begin{proof}
In view of Proposition \ref{prop4.4}(a), it is sufficient to show that ${\rm Ann}(P) = V(\Gamma_E(P))$. On the contrary, suppose that there exists a vertex  ${\rm ann}(x) \in V(\Gamma_E(P)) \setminus {\rm Ann}(P)$. Since the degree  of each vertex in $\Gamma (P)$ is finite,  $P$ is a poset with ACC for annihilators. Therefore,  there exists ${\rm ann}(z) \in {\rm Ann}(P)$ such that ${\rm ann}(x) \subsetneq {\rm ann}(z)$, whence ${\rm deg}(x)+1 \leq {\rm deg}(z)$. This contradicts the fact that $\Gamma (P)$ is an $r$-regular graph, and the proposition is proved.\end{proof}

\begin{prop}\label{prop4.11}
 Let $P$ be a poset with ACC for annihilators. If $|{\rm Ann}(P)| < \infty$,  then $|V(\Gamma_E(P))|$ $<\infty$; in fact, $|V(\Gamma_E(P))| \leq 2^{|{\rm Ann}(P)|} - 2$.
\end{prop}
\begin{proof}
Consider the set $\mathfrak{S}$ of all subsets of ${\rm Ann}(P)$.  Then $|\mathfrak{S}| = 2^{|{\rm Ann}(P)|} < \infty$. Given ${\rm ann}(x)\in \mathfrak{B}=V(\Gamma_E (P))$, let $S_x$ denote the largest subset of ${\rm Ann}(P)$ such that ${\rm ann}(x) \subseteq {\rm ann}(z)$ for all ${\rm ann}(z) \in S_x$; in other words, $S_x = \{{\rm ann}(z) \in {\rm Ann}(P) \mid {\rm ann}(x) \subseteq {\rm ann}(z)\}$. Note that each member of $\mathfrak{B}$ is contained in some member of  ${\rm Ann}(P)$ but no member of $\mathfrak{B}$ is contained in every member of ${\rm Ann}(P)$. Therefore, we have a well-defined map $\psi : V(\Gamma_E (P)) \longrightarrow \mathfrak{S}$ given by ${\rm ann}(x) \mapsto S_x$ such that $\emptyset, {\rm Ann}(P) \notin {\rm Im}(\psi)$. It is now enough to show that $\psi$ is a one one map. So, let ${\rm ann}(x), {\rm ann}(y) \in \mathfrak{B}$ such that $S_x =S_y$.  Let $w \in {\rm ann}(x)\setminus \{0\}$. Then, by Lemma \ref{lem4.3}(c) and the definition of $S_x$,  we have  ${\rm ann}(w) \nsubseteq {\rm ann}(z)$ for each ${\rm ann}(z) \in S_x$. It follows that ${\rm ann}(w) \cup {\rm ann}(y) \nsubseteq {\rm ann}(z)$ for each ${\rm ann}(z) \in S_x = S_y$. Also, by  the maximality of $S_y$, we have ${\rm ann}(w) \cup {\rm ann}(y) \nsubseteq {\rm ann}(z)$ for each ${\rm ann}(z) \in {\rm Ann}(P) \setminus S_y$. Therefore, using Lemma \ref{lem4.3}(c) once again, we have $w \in {\rm ann}(y)$. Thus, ${\rm ann}(x) \subseteq {\rm ann}(y)$. Similarly, we have ${\rm ann}(y) \subseteq {\rm ann}(x)$, and hence, $[x]=[y]$. This completes the proof.\end{proof}

It may noted here that the bound mentioned in the above proposition is the best possible. For example, consider a finite set $X=\{1,2, \dots, n\}$ and define $P$ to be the set of all subsets of $X$ partially ordered under set inclusion with least element $\phi$. It is then easy to see that the vertex set of  $\Gamma (P)$  consists precisely of nontrivial proper subsets of $X$,  ${\rm Ann}(P)= \{\{1\},\{2\}, \dots, \{n\}\}$, and $\Gamma_E (P)\cong \Gamma (P)$. 

 As an immediate consequence of Proposition \ref{prop4.11}, we have the following  result.
\begin{cor}\label{cor4.12}
Let $P$ be a poset. Then $|V(\Gamma_E(P))|<\infty$ if and only if $\omega (P)< \infty$.
\end{cor}
\begin{proof}
If $\omega (P)< \infty$, then, by \cite[Lemma 2.4]{hj}, $P$ is a poset with ACC for annihilators, and so,  it follow from Proposion \ref{prop4.5}  and  Proposition \ref{prop4.11} that  $|V(\Gamma_E(P))|<\infty$. Converse is trivial, since $\omega (P) = \omega (\Gamma_E (P))$.\end{proof}

Note that if P is a poset such that $(x]\cap {\rm Min}(P^\times)\neq \emptyset$ for all $x \in Z (P)^\times$, then $P$ is not necessarily a poset with ACC for annihilators. For example, one may look at the same poset that has been considered in the para preceding Proposition \ref{prop4.2}. However, we have the following small result concerning such posets. 
\begin{prop}\label{prop4.13}
 Let  P be a poset such that $(x]\cap {\rm Min}(P^\times)\neq \emptyset$ for all $x \in Z (P)^\times$. Then   ${\rm Ann}(P)= \{{\rm ann}(z) \mid z \in {\rm Min}(P^\times)\}$.
\end{prop}
\begin{proof}
Let $z \in {\rm Min}(P^\times)$. Suppose that ${\rm ann}(z) \subseteq {\rm ann}(x)$, where ${\rm ann}(x) \in \mathfrak{A}$.   Clearly, $x \notin {\rm ann}(z)$, which means that  $z \leq x$. But then ${\rm ann}(x) \subseteq {\rm ann}(z)$, and so it follows that  ${\rm ann}(z) \in {\rm Max}(\mathfrak{A})= {\rm Ann}(P)$.  Conversely, suppose that ${\rm ann}(x) \in  {\rm Ann}(P)$.  Choose $z \in (x]\cap {\rm Min}(P^\times)$.  Then, ${\rm ann}(x) \subseteq {\rm ann}(z)$, and so it follows from the maximality of ${\rm ann}(x)$ in $\mathfrak{A}$ that ${\rm ann}(x) = {\rm ann}(z)$.  This completes the proof.\end{proof}

We conclude our discussion with the following example.
\begin{exam}\label{exam4.14}
Consider a partially ordered set 
\[
P = \{0\} \cup A \cup ( \underset{k \in \mathbb{N}} \cup B_k),
\]
where $A$ is any set (finite or infinite), $B_k = \{(\frac{1}{n_1}, \frac{1}{n_2}, \dots, \frac{1}{n_k}) \mid n_i \in \mathbb{N} \text{ for } 1 \leq i \leq k \}$, and  the partial order is defined as follows: 
\begin{align*}
a,  \textstyle (\frac{1}{n_1}, \frac{1}{n_2}, \dots, \frac{1}{n_k})\; &> \; 0,\\
\textstyle (\frac{1}{n_1}, \frac{1}{n_2}, \dots, \frac{1}{n_k})\; &> \; \textstyle  (\frac{1}{n_1}, \frac{1}{n_2}, \dots, \frac{1}{n_k},\frac{1}{n_{k+1}}, \dots, \frac{1}{n_{k+t}}),\\
\textstyle \text{and} \quad (\frac{1}{n_1}, \frac{1}{n_2}, \dots, \frac{1}{n_{k-1}}, \frac{1}{n_k})\; &> \; \textstyle  (\frac{1}{n_1}, \frac{1}{n_2}, \dots, \frac{1}{n_{k-1}}, \frac{1}{n_k +1}) 
\end{align*}
for all $a \in A$, and for all $k, t, n_i \in \mathbb{N}$  with $1 \leq i \leq k$.  It is not difficult to see that ${\rm Ann}(P) = \{ {\rm ann}(a) \mid a \in A \}$, the set $\{(\frac{1}{n}, \frac{1}{n}) \mid n \in \mathbb{N}
 \}$ is an infinite clique, and ${\rm ann}((1)) \subsetneq {\rm ann}((\frac{1}{2}))\subsetneq {\rm ann}((\frac{1}{3})) \subsetneq \dots$ is an infinite strictly ascending chain of annihilator ideals  in $P$. Thus, a poset may have finitely many annihilator prime ideals but fail to be a poset with ACC for annihilators. In this example we also have $A= {\rm Min}(P^\times)$ but $(x]\cap {\rm Min}(P^\times)=  \emptyset$ for all $x \in Z (P)^\times \setminus {\rm Min}(P^\times)$, showing that the converse of Proposition \ref{prop4.13} is far from being true.
\end{exam}

\section*{Acknowledgment}
The first author wishes to express his thanks to M. R. Pournaki of Sharif University of Technology, and M. Alizadeh of University of Tehran, both from Iran, for their encouragement and useful suggestions.  The second author is grateful to Council of Scientific and Industrial Research  (India) for its  financial assistance  (File No. 09/347(0209)/2012-EMR-I).


\begin{thebibliography}{33}
\bibitem{admpy}
M. Alizadeh, A.K. Das, H.R. Maimani, M.R. Pournaki, S. Yassemi,  {\em On the diameter and girth of zero-divisor graphs of posets}, Discrete Applied Mathematics  \textbf{160} (2012), 1319 -- 1324.

\bibitem{an}
D. D. Anderson, M. Naseer,  {\em Beck's coloring of a commutative ring}, Journal of Algebra  \textbf{159}(2)(1993), 500 -- 514.

\bibitem{al}
D. F. Anderson, P. S. Livingston,  {\em The zero-divisor graph of a commutative ring}, Journal of Algebra  \textbf{217}(2)(1999), 434 -- 447.

\bibitem{ib}
I. Beck,  {\em Coloring of commutative rings}, Journal of  Algebra  \textbf{116}(1988), 208 -- 226.

\bibitem{co} 
G. Chartrand, O. R. Oellermann,  {\em Applied and Algorithmic  Graph Theory}, McGraw-Hill, Inc., New York, 1993.

\bibitem{dp}
B.A. Davey, H.A. Priestley,  {\em Introduction to Lattices and Order}, second edition, Cambridge University Press, New York, 2002.

\bibitem{rh}
R. Halas,  {\em Annihilators and ideals in ordered sets}, Czechoslovak Mathematical Journal  \textbf{45} (120)(1)(1995), 127 -- 134.

\bibitem{rh2}
R. Halas,  {\em Relative polars in ordered sets}, Czechoslovak Mathematical Journal \textbf{50}(125)(2)(2000), 415 -- 429.

\bibitem{hj} 
R. Halas, M. Jukl,  {\em On Beck's coloring of posets}, Discrete Mathematics  \textbf{309}(13)(2009), 4584 -- 4589.

\bibitem{hl}
R. Halas, H. Langer,  {\em The zerodivisor graph of a qoset}, Order \textbf{27}(3)(2010), 343 -- 351.

\bibitem{vj}
V. Joshi,  {\em Zero divisor graph of a poset with respect to an ideal}, Order  \textbf{29}(3)(2012), 499 -- 506.

\bibitem{lw}
D. Lu, T. Wu,  {\em The zero-divisor graphs of posets and an application to semigroups}, Graphs and Combinatorics   \textbf{26}(6)(2010), 793 -- 804.

\bibitem{mu}
S.B. Mulay,  {\em Cycles and symmetries of zero-divisors}, Communications in Algebra  \textbf{30}(7)(2002), 3533 -- 3558.

\bibitem{nwd}
S. K. Nimbhorkar, M. P. Wasadikar, L. Demeyer,  {\em Coloring of meet-semilattices}, Ars Combinatoria \textbf{84}(2007), 97 -- 104.

\bibitem{sw}
S. Spiroff, C. Wickham,  {\em A zero divisor graph determined by equivalence classes of zero divisors}, Communications in  Algebra  \textbf{39}(7)(2011), 2338 -- 2348.

\bibitem{xl}
Z. Xue, S. Liu,  {\em Zero-divisor graphs of partially ordered sets}, Applied Mathematics Letters \textbf{23}(4)(2010),  449 -- 452.


\end{thebibliography}
\end{document}